\documentclass{amsart}
\usepackage{amsfonts}
\usepackage{amsmath,amsthm}
\usepackage{amsfonts,amssymb}
\usepackage{CJK}
\usepackage{enumerate}

\newtheorem{thm}{Theorem}[section]
\newtheorem{cor}[thm]{Corollary}
\newtheorem{lem}[thm]{Lemma}

\theoremstyle{remark}
\newtheorem{rem}[thm]{Remark}

\numberwithin{equation}{section}

\newcommand{\al}{\alpha}

\def\vz{\varepsilon}
\def\oz{\omega}
\def\lz{\lambda}

\def\az{\alpha}

\def\Gz{\Gamma}

\def\({\Bigl(}
\def \){ \Bigr)}

 \def\a{{\alpha}}

 \def\ab{{\mathbf a}}
 \def\bb{{\mathbf b}}

 \def\kb{{\mathbf k}}

 \def\xb{{\mathbf x}}
 \def\yb{{\mathbf y}}

 \def\RR{{\mathbb R}}

\def\ss{{\Bbb S}^{d}}

\def\bi{\bibitem}

\def\va{\varepsilon}

\def\bi{\bibitem}

\begin{document}
\def\RR{\mathbb{R}}
\def\Exp{\text{Exp}}
\def\FF{\mathcal{F}_\al}

\title[] {Strong equivalences   of approximation numbers and tractability of weighted anisotropic Sobolev embeddings}

\author[]{ JiDong Hao, Heping Wang} \address{ School of Mathematical Sciences, BCMIIS, Capital Normal
University, Beijing 100048,
 China.}
\email{1047695025@qq.com(D. Hao). wanghp@cnu.edu.cn(H. Wang).}

\keywords{Strong equivalences, Tractability, Approximation
numbers, Weighted anisotropic spaces, Analytic Korobov space}
\subjclass[2010]{41A25, 41A63, 65D15, 65Y20}

\begin{abstract} In this paper,
 we study multivariate approximation defined over weighted  anisotropic Sobolev spaces which depend on two
 sequences $\ab=\{\ab_j\}_{j\geq1}$ and $\bb=\{\bb_j\}_{j\geq1}$ of positive numbers.
We  obtain    strong equivalences  of the approximation numbers,
and  necessary and sufficient conditions on $\ab$,
 $\bb$ to achieve various notions of  tractability of
the weighted  anisotropic Sobolev  embeddings.
\end{abstract}

\maketitle
\input amssym.def

\section{Introduction}
\

\begin{CJK*}{GBK}{song}
\end{CJK*}

This paper is  devoted to investigating  sharp constants of
approximation numbers and tractability of embeddings of weighted
anisotropic Sobolev spaces on   $[0,1]^d$  into $L_2([0,1]^d)$.
The approximation numbers of a bounded linear operator
$T:X\rightarrow Y$ between two Banach spaces are defined as
\begin{align*}\label{1.1}
a_n(T:X\rightarrow Y):&=\inf_{{\rm rank} A< n}\sup_{\|x\|_X\leq 1}\|Tx-Ax\|_Y\notag\\
&=\inf_{{\rm rank} A< n}\|T-A\|_{X\rightarrow Y},\ \ n\in \Bbb N.
\end{align*}
 They
describe the best approximation of $T$ by finite rank operators.
If $X$ and $Y$ are Hilbert spaces and $T$ is compact, then $a_n(T
)$ is the $n$-th singular number of $T$. Also $a_n(T)$ is just the
$(n-1)$-th minimal worst-case error  with respect to arbitrary
algorithms and general linear information in the Hilbert setting.

 Recently,  K\"uhn and many other
authors investigated and obtained  strong equivalences,
preasymptotics, asymptotics of the approximation numbers and
tractability of the classical isotropic Sobolev embeddings,
Sobolev embeddings of dominating mixed smoothness,  Gevrey space
embeddings,  anisotropic Sobolev embeddings on the torus $\Bbb
T^d=[0,2\pi]^d$ (see \cite{KSU1, KSU2, KMU, CW}), and Sobolev
embeddings  and Gevrey type  embeddings on the sphere $\ss$ and on
the ball $\Bbb B^d$ (see \cite{CW1}). In \cite{WW} Werschulz and
Wo\'zniakowski investigated tractability of  weighted isotropic
Sobolev embeddings.

We note that unlike on the torus $\Bbb T^d=[0,2\pi]^d$, the
anisotropic spaces $W_2^{\bb}([0,1]^d)$ on the torus $[0,1]^d$
naturally induce weighted norms given in terms of Fourier
coefficients (see \eqref{2.4} in Section 2.1), where
$\bb=\{b_j\}_{j\ge 1}$ is a  positive sequence describing
smoothness index of each variable. In this paper we consider
general  weighted anisotropic Sobolev spaces
$W_2^{\ab,\bb}([0,1]^d)$ on the torus $[0,1]^d$ whose definitions
are given  in Section 2.1, where the positive sequence
$\ab=\{a_j\}_{j\ge1}$ is a scaling parameter sequence, and
$\bb=\{b_j\}_{j\ge 1}$ is a smoothness parameter sequence. We
discuss the approximation numbers and tractability of the weighted
anisotropic Sobolev embedding
\begin{equation}\label{1.1}I_d: \,W_2^{\bf a,b}([0,1]^d)\longrightarrow L_2([0,1]^d),\ \end{equation}  where $I_d$ is the
identity (embedding) operator.

 We obtain strong equivalences of the
approximation numbers $a_n(I_d)\equiv a_n(I_d: \,W_2^{\bf
a,b}([0,1]^d)\to L_2([0,1]^d))$ as $n\to\infty$.  We remark that
the sharp orders  of $a_n(I_d)$ depend only on the smoothness
parameter sequence $\bb$, and the sharp constants are closely
related to the volume of the generalized ellipsoid defined by
$\ab,\bb$ and $d$.   Our result generalizes  Theorem 2.4 in
\cite{CW}. However, we do not obtain results about preasymptotics
of $a_n(I_d)$ as in \cite{CW}.

We also consider  tractability of  the approximation problem
$I=\{I_d\}$ of the weighted anisotropic Sobolev embeddings.  We
consider algorithms that use finitely many continuous linear
functionals. The information complexity $n(\vz, I_d)$ is defined
to be the minimal number of linear functionals which are needed to
find an approximation to within an error threshold $\varepsilon$.
There are two kinds of tractability based on polynomial
convergence and exponential convergence. The classical
tractability describes how the information complexity
 behaves as a function of $d$ and $\vz^{-1}$, while the
exponential convergence-tractability (EC-tractability) does as one
of $d$ and $(1+\ln\vz^{-1})$. Nowadays the  study of tractability
and EC-tractability   has attracted much interest, and a great
number of interesting results
 have been obtained (see
\cite{NW1, NW2, NW3, DLPW, IKPW, PP, X2} and the references
therein).

Denote by $H(K_{d,\ab,2\bb})$ the analytic Korobov space which is
a reproducing kernel Hilbert space with the reproducing kernel
$K_{d,\ab,2\bb}$, and   whose definition will be given in Section
2.2.   Such  space $H(K_{d,\ab,2\bb})$ has been widely
investigated in the study of tractability and EC-tractability (see
\cite{DKPW, DLPW, IKPW, KPW0, KPW, LX1, Wangh}). Particularly, the
papers \cite{DKPW, Wangh} considered different notions of
EC-tractability of  the approximation problems ${\rm APP}=\{{\rm
APP}_d\}_{d\in\Bbb N}$, and obtained
 the corresponding necessary and sufficient conditions, where
\begin{equation}\label{1.2}{\rm APP}_d: \,H(K_{d,\ab,2\bb})\longrightarrow L_2([0,1]^d),\ \  {\rm with}\ {\rm APP}_d(f)=f.\end{equation}

In this paper, we establish the relationship of the information
complexities $n(\vz,I_d)$ and  $n(\vz,{\rm APP}_d)$. Based on this
relationship, we obtain
 necessary and sufficient conditions for various notions of tractability of the approximation problem $I=\{I_d\}_{d\in \Bbb N}$.

The paper is organized as follows. In Section 2 we introduce
 the weighted anisotropic Sobolev spaces,
 the analytic  Korobov spaces,   properties of the approximation numbers,  tractability, and then state out main results.
 Section 3 is devoted to proving
 strong equivalence  of the
approximation numbers of the weighted anisotropic embeddings. In
Section 4 we prove tractability of the weighted anisotropic
embeddings.

\section{Preliminaries and main results}

\subsection{ Weighted anisotropic Sobolev spaces on  $[0,1]^d$}

\

 Denote
by $ L_2([0,1]^d)$ the collection of  measurable functions $f$ on
$[0,1]^d$ with finite norm
$$\|f\|_2=\Big(\int_{[0,1]^d}|f(\xb)|^2\, d\xb\Big)^\frac{1}{2}<+ \infty.$$
It is well known that any $f\in L_2([0,1]^d)$ can be expressed by
its Fourier series $$f(\xb)=\sum\limits_{\kb\in \Bbb
Z^d}\hat{f}(\kb)e^{2\pi i\,\kb\xb},\ \ \ \xb\in[0,1]^d,$$ where
$i=\sqrt{-1}$, $\ \kb\xb=\sum_{j=1}^d {k_jx_j}$, and
$$\hat{f}(\kb)=\int_{[0,1]^d}f(\xb)e^{-2\pi
i\kb\xb}d\xb,\quad \kb\in\ \Bbb Z^d,$$are the Fourier coefficients
of the function $f$. We have the following Parseval equality:
$$\|f\|_2=\Big(\sum_{\kb\in \Bbb
Z^d}|\hat{f}(\kb)|^2\Big)^{1/2}.$$

For $r>0$, denote by $D_j^{r}f\equiv \frac{\partial^r}{\partial
x_j^r}f$ the $r$-order partial derivative of $f$ with respect to
$x_j$ in the sense of Weyl, i.e.,  $$D_j^rf(\xb)=\sum_{\kb\in\,
{\Bbb Z}^{d}}(
 2\pi ik_{j})^{r}\hat{f}({\bf k})e^{2\pi i\,{\bf k}{\bf x}},\ \
 (2\pi ik_{j})^{r}=|2\pi k_{j}|^{r}\exp\,(\frac{r \pi i}{2}{\rm sign}\,k_{j}).$$
It follows from the Parseval equality that for $D_j^{r}f\in
L_2([0,1]^d)$,
\begin{equation}\label{2.1}\|D_j^rf\|_2=\Big(\sum_{\kb\in \Bbb
Z^d}|2\pi k_j|^{2r} |\hat{f}(\kb)|^2\Big)^{1/2}.\end{equation}

Now we define  weighted anisotropic Sobolev spaces. Let ${\bf
a}=\big\{a_k\big\}_{k\ge1}$  and ${\bf b}=\big\{b_k\big\}_{k\ge1}$
be two sequences of positive  numbers.  Usually, we assume that
the sequences ${\bf a}=\big\{a_k\big\}_{k\ge1}$  and ${\bf
b}=\big\{b_k\big\}_{k\ge1}$ satisfy
\begin{equation}\label{2.2}0<a_1\le a_2 \le \dots\le a_k\le \dots ,\ \ \ {\rm and}\  \ \ b_*:=\inf_{k\ge 1}
b_k>0.\end{equation}

The weighted anisotropic Sobolev space $W_2^{\ab,\bb}([0,1]^d)$ is
defined by
\\ $$W_2^{\ab,\bb}([0,1]^d)=\Big\{f \in L_2([0,1]^d)\ :\ D_j^{{b_j}}f\in L_2([0,1]^d),\ \ j=1,2,\dots,d\Big\},$$ \\
with norm
\begin{align*}\big\|f\big\|_{W_2^{\ab,\bb}}&=\Big(\big\|f\big\|_2^2+\sum_{j=1}^d\frac{a_j}{(2\pi)^{2b_j}}\big\| D_j^{{b_j}}f\big\|_2^2\Big)^{1/2}. \end{align*}
Clearly,  $W_2^{\ab,\bb}([0,1]^d)$ is a Hilbert space. We remark
that  ${\bf b}$ is a smoothness parameter sequence, ${\bf a}$ is a
(regulated) scaling parameter sequence with respect to the
sequence  ${\bf b}$.

It follows from \eqref{2.1} that
\begin{equation}\label{2.3}\big\|f\big\|_{W_2^{\ab,\bb}}=\Big(
\sum\limits_{\kb\in \Bbb Z^d}\big(1+\sum_{j=1}^d a_j\vert
k_j\vert^{2b_j}\big)
\vert\hat{f}(\kb)\vert^2\Big)^{\frac{1}{2}}.\end{equation}

 If
$a_j=(2\pi)^{2b_j},\ j\in\Bbb N$, then $W_2^{\ab,\bb}([0,1]^d)$
recedes to the usual  anisotropic Sobolev spaces
$W_2^{\bb}([0,1]^d)$ on the torus $[0,1]^d$. We emphasize  that
the anisotropic Sobolev spaces given in \cite{CW} are defined on
the torus $\Bbb T^d=[0,2\pi]^d$, not on $[0,1]^d$.  It is easily
seen that
\begin{equation}\label{2.4}\big\|f\big\|_{W_2^{\bb}}=\Big(\big\|f\big\|_2^2+\sum_{j=1}^d\big\| D_j^{{b_j}}f\big\|_2^2\Big)^{1/2}=\Big(
\sum\limits_{\kb\in \Bbb Z^d}\big(1+\sum_{j=1}^d \vert 2\pi
k_j\vert^{2b_j}\big)
\vert\hat{f}(\kb)\vert^2\Big)^{\frac{1}{2}}.\end{equation}

We emphasize that if we denote $\tilde \bb=\{\tilde b_j\}$,
$\tilde b_j=(2\pi)^{2b_j},\ j\in\Bbb N$, then
$$W_2^{\bb}([0,1]^d)=W_2^{\tilde\bb,\bb}([0,1]^d).$$

\subsection{Analytic Korobov spaces}

\

Let $\ab=\{\ab_j\}_{j\geq1}$ and $\bb=\{\bb_j\}_{j\geq1}$ be the
sequences  satisfying  \eqref{2.2}. Fix $\oz\in (0,1)$. We define
the analytic Korobov kernel $K_{d,\ab,2\bb}$  by
$$K_{d,\ab,2\bb}(\xb,\yb)=\sum\limits_{\kb\in \Bbb Z^d}\oz_{\kb}e^{2\pi i\kb (\xb-\yb)}, \ \ {\rm for} \ \ {\rm all} \ \ \xb, \yb\in [0,1]^d,$$
where $$\oz_{\kb}=\oz^{\sum_{j=1}^d a_j\vert k_j\vert^{2b_j}}, \ \
{\rm for} \ \ {\rm all} \ \ \kb\in \Bbb Z^d.$$
 Denote by $H(K_{d,\ab,2\bb})$ the analytic  Korobov space which is a
reproducing kernel Hilbert space with the reproducing kernel
$K_{d,\ab,2\bb}$.   The inner product of  the space
$H(K_{d,\ab,2\bb})$ is given by
$$\langle f,g\rangle_{H(K_{d,\ab,2\bb})}=\sum\limits_{\kb\in\Bbb
Z^d}\hat{f}(\kb)\overline{\hat{g}(\kb)}\oz_{\kb}^{-1}, \ \ \ f,g
\in H(K_{d,\ab,2\bb}),$$ where $\hat{f}(\kb),\ \hat g(\kb),\
\kb\in \Bbb Z^d$ are the
 Fourier coefficients of the functions $f$ and $g$. The norm of  a function $f$ in $H(K_{d,\ab,2\bb})$
is given by
$$\Vert f\Vert_{H(K_{d,\ab,2\bb})}=(\sum\limits_{\kb\in \Bbb
Z^d}\vert\hat{f}(\kb)\vert^2\oz_{\kb}^{-1})^{\frac{1}{2}}<\infty.$$
Obviously, $\{e_{\kb}\}_{\kb\in \Bbb Z^d}$ is an orthonormal basis
for  $H(K_{d,\ab,2\bb})$ with $$e_{\kb}(\xb)=e^{2\pi
i\kb\xb}\oz_{\kb}^{\frac{1}{2}}.$$

\subsection{Approximation numbers}

\

 Let
${H}$ and ${G}$ be two  Hilbert spaces and $T$ be a compact linear
operator from  $H$ to $G$. The basic properties of approximation
numbers $a_n(T: H\to G)$ are well known, see e.g., Pietsch
\cite[Chapter 11]{Pi1} and \cite[Chapter 2]{Pi2}, K\"onig
\cite[Section 1.b]{K}, Pinkus \cite[Theorem IV.2.2]{P}, and Novak
and Wo\'zniakowski \cite[Corollary 4.12]{NW1}.

Let $H$ be a separable Hilbert space, $\{e_k\}_{k=1}^\infty $ an
orthonormal basis  in $H$, and ${\tau}=\{\tau_k\}_{k=1}^\infty$ a
sequence of positive numbers with $$\tau_1\ge\tau_2\ge \dots\ge
\tau_k\ge\cdots>0.$$ Let $H^{ \tau}$ be a Hilbert space defined by
$$H^\tau=\Big\{x\in H\ :\ \|x\|_{H^\tau}=\Big(\sum_{k=1}^\infty
\frac{|(x,e_k)|^2}{\tau_k^2}\Big)^{1/2}<\infty\Big\}.$$ According
to \cite[Corollary 2.6]{P} we have the following lemma.

\begin{lem}  Let $H, \tau $ and $H^{\tau }$ be defined as above.
Then $$a_n(I_d: H^\tau\to H)= \tau_n, \
 \ n\in\Bbb N.$$
\end{lem}

Let $\{W_{\ab,\bb,d}^*(l)\}_{l=1}^\infty$ be the non-increasing
rearrangement of$$\Big\{\Big(1+\sum_{j=1}^d a_j\vert
k_j\vert^{2b_j}\Big)^{-\frac{1}{2}}\Big\}_{\kb=(k_1,\dots,k_d)\in
\Bbb Z^d},$$ and let $\{\lz_{d,k}\}_{k=1}^\infty$ be  the
non-increasing rearrangement of $$\{\omega_{\kb}\}_{\kb\in \Bbb
Z^d}=\Big\{ \oz^{\sum_{j=1}^d a_j\vert
k_j\vert^{2b_j}}\Big\}_{\kb=(k_1,\dots,k_d)\in \Bbb Z^d},$$with
fixed $\oz\in (0,1)$. According to Lemma 2.1, we get
\begin{equation}\label{2.5}a_n(I_d:W_2^{\ab,\bb}([0,1]^d)\rightarrow L_2([0,1]^d))=W_{\ab,\bb,d}^*(n),\end{equation}
and
\begin{equation}\label{2.6}a_n({\rm APP}_d: H(K_{d,\ab,2\bb})\rightarrow
L_2([0,1]^d))=\lz_{d,n}^{1/2}.\end{equation}

\subsection{General notations of tractability}
\

Let $H_d$ and $G_d$ be two sequences of Hilbert space and for each
$d\in \Bbb N$, $F_d$ be the unit ball of $H_d$. Assume a sequence
of bounded linear operators (solution operators)
$$S_d:H_d\rightarrow G_d$$
for all $d\in \Bbb N$. For $n\in \Bbb N$ and $f\in F_d$, $S_d f$
can be approximated by algorithms $$A_{n,d}(f)=\phi
_{n,d}(L_1(f),L_2(f),\dots ,L_n(f)),$$ where $L_j,\ j=1,2,\dots
,n$ are continuous linear functionals on $F_d$ which are called
general information, and $\phi_{n,d}: \Bbb R^n\rightarrow G_d$ is
an arbitrary mapping. The worst case error $e(A_{n,d})$ of the
algorithm $A_{n,d}$ is defined as
$$e(A_{n,d})=\sup\limits_{f\in F_d}\Vert
S_d(f)-A_{n,d}(f)\Vert_{G_d}.$$ Furthermore, we defined the $n$th
minimal worst-case error as
$$e(n,S_d)=\inf\limits_{A_{n,d}}e(A_{n,d}),$$                  where
the infimum is taken over all algorithms using n information
operators $L_1,L_2,\dots ,L_n$. For $n=0$, we use $A_{0,d}=0$. The
error of $A_{0,d}$ is called the initial error and is given
by$$e(0,S_d)=e(A_{0,d})=\sup\limits_{f\in F_d}\Vert
S_d(f)\Vert_{G_d}.$$

From  \cite[p. 118]{NW1}, we know that the $n$th minimal
worst-case error $e(n,S_d)$ with respect to arbitrary algorithms
and general information in the Hilbert setting is just the
$(n+1)$-approximation number $a_{n+1}(S_d:H_d\rightarrow G_d)$,
i.e.,
$$e(n,S_d)=a_{n+1}(S_d:H_d\rightarrow G_d).$$

In this paper, we consider the embedding operators $S_d=I_d$ and
$S_d=APP_d$ which are defined by \eqref{1.1} and \eqref{1.2}.  We
note that
$$e(0,I_d)=\Vert I_d\Vert=1\ \ \ {\rm and}\ \ \ e(0,{\rm
APP}_d)=\Vert APP_d\Vert=1.$$ In both cases, $e(0,S_d)=1$. In
other words, the normalized error criterion and the absolute error
criterion coincide for the approximation problems $I=\{I_d\}$ and
${\rm APP}=\{{\rm APP}_d\}$.

For $\varepsilon \in (0,1)$ and $d\in\Bbb N$, let $n(\varepsilon
,S_d)$ be the information complexity defined by
\begin{equation}\label{2.6-1}n(\varepsilon ,S_d)=\min\{n\ :\
e(n,S_d)\leq\varepsilon\}, \end{equation}where
$$e(n,I_d)=a_{n+1}(I_d: \,W_2^{\bf a,b}([0,1]^d)\rightarrow L_2([0,1]^d)),$$ $$ e(n, {\rm APP}_d)= a_{n+1}({\rm APP}_d:
\,H(K_{d,\ab,2\bb})\rightarrow L_2([0,1]^d)).$$

 Now, we list the concepts of tractability below. We say
that the approximation problem $S=\{S_d\}_{d\in\Bbb N}$ is

$\bullet$   {\it strongly polynomially tractable (SPT)}   iff
there exist non-negative numbers $C$ and $p$ such that for all
$d\in \Bbb N,\ \va \in (0,1)$,
\begin{equation*}
n(\va ,S_d)\leq C(\va ^{-1})^p;
\end{equation*}

 $\bullet$  {\it polynomially tractable
(PT)} iff there exist non-negative numbers $C, p$ and $q$ such
that for all $d\in \Bbb N, \ \va \in(0,1)$,
\begin{equation*}
n(\va ,S_d)\leq Cd^q(\va ^{-1})^p;
\end{equation*}

$\bullet$   {\it quasi-polynomially tractable (QPT)} iff there
exist two constants $C,t>0$ such that for all $d\in \Bbb N, \ \va
\in(0,1)$,
\begin{equation*}
n(\va ,S_d)\leq C\exp[t(1+\ln\va ^{-1})(1+\ln d)];
\end{equation*}

$\bullet$ {\it uniformly weakly tractable (UWT)} iff for all
$\alpha,\beta>0$,
\begin{equation*}
\lim_{\varepsilon ^{-1}+d\rightarrow \infty }\frac{\ln n(\va
,S_d)}{(\va ^{-1})^{\alpha }+d^{\beta }}=0;
\end{equation*}

$\bullet$  {\it weakly tractable (WT)} iff
\begin{equation*}
\lim_{\va ^{-1}+d\rightarrow \infty }\frac{\ln n(\va ,S_d)}{\va
^{-1}+d}=0;
\end{equation*}

$\bullet$ {\it $(s,t)$-weakly tractable ($(s,t)$-WT)} for fixed
positive numbers $s$ and $t$ iff
\begin{equation*}
\lim_{\varepsilon ^{-1}+d\rightarrow \infty }\frac{\ln n(\va
,S_d)}{(\va ^{-1})^{s }+d^{t }}=0.
\end{equation*}

In the above definitions of SPT, PT, QPT, UWT, WT, and $(s,t)$-WT,
if we replace $\frac1{\vz}$ by $(1+\ln \frac 1{\vz})$, we get the
definitions of \emph{exponential convergence-strong polynomial
tractability (EC-SPT)}, \emph{exponential convergence-polynomial
tractability (EC-PT)}, \emph{exponential
convergence-quasi-polynomial tractability (EC-QPT)},
\emph{exponential convergence-uniform weak tractability
 (EC-UWT)}, \emph{exponential convergence-weak tractability
 (EC-WT)}, and \emph{exponential convergence-$(s,t)$-weak tractability
 (EC-$(s,t)$-WT)}, respectively. We now give the above notions of EC-tractability in
 detail.

 \vskip 2mm

 We say that $S=\{S_d\}_{d\in\Bbb N}$ is

 $\bullet$   {\it Exponential convergence-strong polynomial
tractable  (EC-SPT)}   iff there exist non-negative numbers $C$
and $p$ such that for all $d\in \Bbb N,\ \va \in (0,1)$,
\begin{equation*}
n(\va ,S_d)\leq C(1+\ln \va ^{-1})^p;
\end{equation*}

 $\bullet$  {\it Exponential convergence-polynomial
tractable  (EC-PT)} iff there exist non-negative numbers $C, p$
and $q$ such that for all $d\in \Bbb N, \ \va \in(0,1)$,
\begin{equation*}
n(\va ,S_d)\leq Cd^q(1+\ln \va ^{-1})^p;
\end{equation*}

$\bullet$   {\it Exponential convergence-quasi-polynomial
tractable (EC-QPT)} iff there exist two constants $C,t>0$ such
that for all $d\in \Bbb N, \ \va \in(0,1)$,
\begin{equation*}
n(\va ,S_d)\leq C\exp\{t[1+\ln(1+\ln \va ^{-1})](1+\ln d)\};
\end{equation*}

$\bullet$ {\it Exponential convergence-uniformly weakly tractable
(EC-UWT)} iff for all $\az, \beta>0$,
\begin{equation*}
\lim_{\varepsilon ^{-1}+d\rightarrow \infty }\frac{\ln n(\va
,S_d)}{(1+\ln \va ^{-1})^{\az }+d^{\beta }}=0;
\end{equation*}

$\bullet$  {\it Exponential convergence-weakly tractable (EC-WT)}
iff
\begin{equation*}
\lim_{\va ^{-1}+d\rightarrow \infty }\frac{\ln n(\va
,S_d)}{(1+\ln\va ^{-1})+d}=0;
\end{equation*}

$\bullet$ {\it Exponential convergence-$(s,t)$-weakly tractable
(EC-$(s,t)$-WT)} for fixed positive $s$ and $t$ iff
\begin{equation*}
\lim_{\varepsilon ^{-1}+d\rightarrow \infty }\frac{\ln n(\va
,S_d)}{(1+\ln \va ^{-1})^{s }+d^{t }}=0.
\end{equation*}

\subsection{Main results}

\

 Let ${\bf
a}=\big\{a_k\big\}_{k\ge1}$ and ${\bf b}=\big\{b_k\big\}_{k\ge1}$
be two sequences of positive  numbers.  For $t>0$ and $d\in\Bbb
N$, denote by
$$B_{\bf a,b}^d(t)=\{x\in \Bbb R^d: \sum_{j=1}^d\a_j\vert
x_j\vert^{b_j}\leq t\}$$ the generalized ellipsoid  in $\Bbb R^d$.
We write  $B_{\bf a,b}^d$ instead of $B_{\bf a,b}^d(1)$ for
brevity. Clearly, when $a_1=a_2=\cdots=a_d=1$, $B_{\bf a,b}^d$
recedes to the generalized unit ball $$B_{\bf b}^d=\Big\{x\in \Bbb
R^d: \sum_{j=1}^d\vert x_j\vert^{b_j}\leq 1\Big\}.$$ We shall show
that the volume of the generalized ellipsoid $B_{\bf a,b}^d$ is
\begin{align*}{\rm vol}(B_{\bf a,b}^d)&=2^d a_1^{-\frac{1}{b_1}}
a_2^{-\frac{1}{b_2}}\cdots a_d^{-\frac{1}{b_d}}
\frac{\Gz(1+\frac{1}{b_1})\Gz(1+\frac{1}{b_2})\cdots\Gz(1+\frac{1}{b_d})}{\Gz(1+\frac{1}{b_1}+\frac{1}{b_2}+\cdots+\frac{1}{b_d})}\\
&=a_1^{-\frac{1}{b_1}} a_2^{-\frac{1}{b_2}}\cdots
a_d^{-\frac{1}{b_d}}\,{\rm vol}(B_{\bf b}^d),\end{align*} where
$\Gz(x)=\int_0^\infty t^{x-1}e^{-t}dt$ is the Gamma function. The
volume of $B_\bb^d$ is known (see \cite{Wangx}).

 The authors in \cite{CW} investigated, among others, strong equivalence of
approximation numbers  $a_n(I_d:W_2^{\bf R}(\Bbb T^d)\rightarrow
L_2(\Bbb T^d))$, where $\Bbb T^d=[0,2\pi]^d$, ${\bf
R}=(R_1,R_2,\dots,R_d)\in \Bbb R_+^d$, and $I_d$ is the identity
(embedding) operator. They obtained (see \cite[Theorem 2.4]{CW})
\begin{equation}\label{2.7}\lim_{n\to\infty} n^{g({\bf R})}a_n(I_d:W_2^{\bf R}(\Bbb T^d)\rightarrow
L_2(\Bbb T^d))=({\rm vol}(B_{2{\bf R}}^d))^{g({\bf
R})},\end{equation} where $g({\bf
R})=\frac{1}{\frac{1}{R_1}+\frac{1}{R_2}+\cdots+\frac{1}{R_d}}.$

In this paper, we generalize the above result to the weighted
anisotropic spaces $W_2^{\ab,\bb}([0,1]^d)$. We use the volume
argument to obtain the asymptotic behavior of
$a_n(I_d:W_2^{\ab,\bb}([0,1]^d)\rightarrow L_2([0,1]^d)).$ Our
result can be formulated as follows.
\begin{thm}Let $\ab=\{a_j\}_{j\geq1}$ and $\bb=\{b_j\}_{j\geq1}$ be two sequences of positive numbers. Then we have
\begin{equation}\label{2.8}\lim_{n\to\infty} n^{{g_d(\bb)}}
a_n(I_d:W_2^{\ab,\bb}([0,1]^d)\rightarrow L_2([0,1]^d))=({\rm
vol}(B_{\ab,2\bb}^d))^{{g_d(\bb)}},\end{equation} where
$g_d(\bb)=\frac{1}{\frac{1}{b_1}+\frac{1}{b_2}+\cdots+\frac{1}{b_d}}$.
\end{thm}

In particular, we have the following corollary.

\begin{cor}Let  $\bb=\{\bb_j\}_{j\geq1}$ be a sequences of positive numbers. Then
\begin{equation}\label{2.9}\lim_{n\to\infty} n^{{g_d(\bb)}}
a_n(I_d:W_2^{\bb}([0,1]^d)\rightarrow
L_2([0,1]^d))=\big((2\pi)^{-2d}{\rm
vol}(B_{2\bb}^d)\big)^{{g_d(\bb)}}.\end{equation}
\end{cor}

\begin{rem}Let $\ab=\{a_j\}_{j\ge1}$ and ${\bf b}=\{b_j\}_{j\ge1}$ be two sequences of positive numbers,  and    $g_d({\bf b})=\frac{1}{1/b_1+\cdots+1/b_d}$.
Theorem 2.2 indicates that the exact decay rate in $n$ of the
approximation numbers $ a_n(I_d:W_2^{\ab,\bb}([0,1]^d)\rightarrow
L_2([0,1]^d))$ is $n^{-g_d(\bb)}$ which is independent of $\ab$,
and the sharp constant is $({\rm
vol}(B_{\ab,2\bb}^d))^{{g_d(\bb)}}$. We can rephrase \eqref{2.8}
 as  strong equivalences
$$a_n\big(I_d :W_2^{\bf a,b}([0,1]^d)\rightarrow L_2 ([0,1]^d)\big)\sim   n^{-g_d({\bf b})}(\mathrm{vol}(B_{\ab,2{\bf
b}}^d))^{g_d({\bf b})}.
$$ The novelty of Theorem 2.2  is that they give strong equivalences and provide
asymptotically optimal (sharp) constants, for arbitrary fixed $d$,
${\bf a}$, and $\bb$.
\end{rem}

 \begin{rem}Comparing \eqref{2.7} with \eqref{2.9}, we get that the sharp
 constants of the approximation numbers of anisotropic Sobolev embeddings on the torus $[a,b]^d$ depend on the volume of the torus. We can show that
\begin{equation*}\lim_{n\to\infty} n^{{g_d(\bb)}}
a_n(I_d:W_2^{\bb}([a,b]^d)\rightarrow
L_2([a,b]^d))=\Big(\Big(\frac{b-a}{2\pi}\Big)^{2d}{\rm
vol}(B_{2\bb}^d)\Big)^{{g_d(\bb)}}.\end{equation*}
  \end{rem}

\begin{rem}According to \cite[Theorem 3.1]{CKS}, we know that the condition $g_d(\bb)>\frac{1}{2}$ is a sufficient
and necessary condition for the embedding from
$W_2^{\ab,\bb}([0,1]^d)$ into $L_\infty([0,1]^d)$  or
$C([0,1]^d)$. Furthermore, for $g_d(\bb)>\frac{1}{2}$, using the
proof technique of  \cite[Theorem 4.3]{CKS} we can show
\begin{equation*}\lim_{n\to\infty}n^{{g_d(\bb)}-\frac{1}{2}}a_n(I_d:W_2^{\ab,\bb}([0,1]^d)
\rightarrow L_\infty([0,1]^d))=(2g_d(\bb)-1)^{-\frac{1}{2}}({\rm
vol}(B_{\ab,2\bb}^d))^{{g_d(\bb)}}.\end{equation*} Note that the
above equality also  holds  if we replace $L_\infty([0,1]^d)$ by
$C([0,1]^d)$.
\end{rem}

Next, we establish a relationship of the information complexities
$n(\vz,I_d)$ and $n(\vz, {\rm APP}_d)$. Such relationship is
crucial for obtaining sufficient and necessary conditions of
various notions of tractability of the approximation problem
$I=\{I_d\}$.

\begin{thm}For  $\varepsilon \in(0,1), \ \ d\in\Bbb N,$  we have  \begin{equation}\label{2.11}n(\varepsilon,{\rm APP}_d)=
n\Big(\Big(\frac{\ln \varepsilon^{-2}}{\ln
\oz^{-1}}+1\Big)^{-\frac{1}{2}},I_d\Big)\end{equation} and
\begin{equation}\label{2.12}n(\varepsilon,I_d)=n(\oz^{\frac{\varepsilon^{-2}-1}{2}},{\rm
APP}_d),\end{equation}where $n(\vz, S_d)$ is given in
\eqref{2.6-1}, $S_d=I_d$ or ${\rm APP}_d$.
\end{thm}

We know that the classical tractability of the multivariate
problem APP has been  solved completely in \cite{KPW, LX1, IKPW}.
For the EC-tractability of  APP,   the sufficient and necessary
conditions for  EC-SPT, EC-PT, EC-QPT, EC-UWT, EC-WT, and
EC-$(s,t)$-WT with $\max(s,t)>1$
   were given in
  \cite{IKPW}, and for EC-$(s,t)$-WT with $\max(s,t)\le 1$ and $\min(s,t)<1$ in \cite{Wangh}. Based on Theorem 2.7 we obtain the following tractability results
of the   approximation problem $I=\{I_d\}$.

\begin{thm}Consider the   approximation problem
$I=\{I_d\}$  in the worst case setting with the sequences ${\bf
a}$ and ${\bf b}$
 satisfying \eqref{2.2}. Then

$(\romannumeral 1)$  SPT holds iff PT holds iff
$$\sum_{j=1}^\infty b_j^{-1}<\infty  \qquad {\rm and}\qquad\underset{j\to\infty}{\underline{\lim}}\frac{\ln a_j}{
j}>0;$$

$(\romannumeral 2)$  QPT holds iff $$\sup_{d\in \Bbb
N}\frac{\sum_{j=1}^d b_j^{-1}}{1+\ln d}<\infty\qquad {\rm
and}\qquad
 \underset{j\to\infty}{\underline{\lim}}\frac{(1+\ln j)\ln a_j}{
j}>0;$$

$(\romannumeral 3)$ UWT holds iff $$\lim_{j\to\infty}\frac{\ln
a_j}{\ln j}=\infty;$$

$(\romannumeral 4)$ WT  holds iff  $$\lim\limits_{j\to
\infty}\frac {a_j}j=\infty;$$

$(\romannumeral 5)$ $(s,t)$-WT with $\max(s/2,t)>1$ always holds;

\vskip 1mm

$(\romannumeral 6)$ $(2,1)$-WT  holds iff
\begin{equation*}\lim\limits_{j\to\infty}a_j=\infty;\end{equation*}

$(\romannumeral 7)$ $(2,t)$-WT with $t<1$ holds iff
\begin{equation*}\lim\limits_{j\to\infty}\frac{ a_j}{\ln j}=\infty;\end{equation*}

$(\romannumeral 8)$  $(s,t)$-WT with $s<2$ and $t\leq 1$  holds
iff
\begin{equation*}\lim\limits_{j\to\infty}\frac{ a_j}{j^{(2-s)/s}}=\infty.\end{equation*}
\end{thm}

In particular, let $\bb=\{b_j\}$ satisfy
\begin{equation}\label{2.13}0<b_1\le b_2\le \dots\le b_j\le
\dots,\end{equation} and $\tilde \bb=\{\tilde b_j\}$, $\tilde
b_j=(2\pi)^{2b_j},\ j\in\Bbb N$. Then
$W_2^{\bb}([0,1]^d)=W_2^{\tilde\bb,\bb}([0,1]^d).$ By Theorem 2.8
we have the following result.

\begin{thm}Consider the   approximation problem
$\tilde I=\{\tilde I_d\}_{d\in\Bbb N}$  in the worst case setting
with the sequence ${\bf b}$
 satisfying \eqref{2.13}, where $$\tilde I_d: W_2^{\bb}([0,1]^d)\to L_2([0,1]^d)\ \ \ {\rm with} \ \ \ \tilde I_d(f)=f.$$
 Then we have
\vskip 2mm

 $(\romannumeral 1)$ $\tilde I$ is  SPT  iff $\tilde I$
is PT  iff  $$\sum\limits_{j=1}^\infty b_j^{-1}<\infty ;$$

 $(\romannumeral 2)$  $\tilde I$ is QPT  iff  $$\sup\limits_{d\in \Bbb
N}\frac{\sum_{j=1}^d b_j^{-1}}{1+\ln d}<\infty;$$

\vskip 2mm

$(\romannumeral 3)$ UWT holds iff
$$\lim_{j\to\infty}\frac{b_j}{\ln j}=\infty;$$

$(\romannumeral 4)$ WT  holds iff  $$\lim\limits_{j\to
\infty}\frac {(2\pi)^{2b_j}}j=\infty;$$

$(\romannumeral 5)$ $(s,t)$-WT with $\max(s/2,t)>1$ always holds;

\vskip 1mm

$(\romannumeral 6)$ $(2,1)$-WT  holds iff
\begin{equation*}\lim\limits_{j\to\infty}b_j=\infty;\end{equation*}

$(\romannumeral 7)$ $(2,t)$-WT with $t<1$ holds iff
\begin{equation*}\lim\limits_{j\to\infty}\frac{(2\pi)^{2b_j}}{\ln j}=\infty;\end{equation*}

$(\romannumeral 8)$  $(s,t)$-WT with $s<2$ and $t\leq 1$  holds
iff
\begin{equation*}\lim\limits_{j\to\infty}\frac{(2\pi)^{2b_j}}{j^{(2-s)/s}}=\infty.\end{equation*}
\end{thm}

\section{Strong equivalences of  approximation numbers  }

 This section is devoted to studying strong equivalence of
the approximation numbers
$a_n(I_d:W_2^{\ab,\bb}([0,1]^d)\rightarrow L_2([0,1]^d))$, where
$\ab=\{a_j\}_{j\ge1}$ and $\bb=\{b_j\}_{j\ge1}$ are two sequences
of positive numbers. In this section we do not need to assume that
$\ab,\ \bb$ satisfy \eqref{2.2}.

We know from \cite{Wangx}  that the volume of the generalized unit
ball $B_\bb^d$ is
$${\rm vol}(B_\bb^d) ={\rm vol}\{x\in \Bbb R^d: \sum_{j=1}^d\vert
x_j\vert^{b_j}\leq
1\}=2^d\frac{\Gz(1+\frac{1}{b_1})\Gz(1+\frac{1}{b_2})\cdots\Gz(1+\frac{1}{b_d})}{\Gz(1+\frac{1}{b_1}+\frac{1}{b_2}+\cdots+\frac{1}{b_d})}.$$

\begin{lem}Let $\ab=\{a_j\}_{j\geq1}$ and $\bb=\{b_j\}_{j\geq1}$ be two sequences of positive numbers. Then
\begin{equation}\label{3.1}{\rm vol}(B_{\bf a,b}^d)=a_1^{-\frac{1}{b_1}}a_2^{-\frac{1}{b_2}}\cdots a_d^{-\frac{1}{b_d}}\, {\rm vol}(B_\bb^d).\end{equation}
\end{lem}

\begin{proof}We make a change of variables $$y_1=x_1a_1^{\frac{1}{b_1}}, \ \  y_2=x_2a_2^{\frac{1}{b_2}},\quad \cdots,\quad y_d=x_da_d^{\frac{1}{b_d}}$$
 that deforms $B_{\bf a,b}^d$ into $B_\bb^d$. The Jacobian determinant is $J(y)={a_1^{-\frac{1}{b_1}}a_2^{-\frac{1}{b_2}}\cdots a_d^{-\frac{1}{b_d}}}$.
 By the change of variables formula, we get $${\rm vol}(B_{\bf a,b}^d)
 =\int_{B_{\bf a,b}^d}1dx=\int_{B_\bb^d}J(y)dy={a_1^{-\frac{1}{b_1}}a_2^{-\frac{1}{b_2}}\cdots a_d^{-\frac{1}{b_d}}}\,{\rm vol}(B_\bb^d).$$ The proof of Lemma 3.1 is finished.
\end{proof}

\begin{lem}Let $\ab=\{a_j\}_{j\geq1}$ and $\bb=\{b_j\}_{j\geq1}$ be two sequences of positive numbers. Then
\begin{equation}\label{3.2}{\rm vol}(B_{\bf a,b}^d(t))=t^{\frac{1}{b_1}+\frac{1}{b_2}+
\cdots+\frac{1}{b_d}}\,{\rm vol}(B_{\bf
a,b}^d)=t^{\frac1{g_d(\bb)}}\,{\rm vol}(B_{\bf
a,b}^d),\end{equation}where
$g_d(\bb)=\frac{1}{\frac{1}{b_1}+\frac{1}{b_2}+\cdots+\frac{1}{b_d}}$.
\end{lem}

\begin{proof}We make a change of variables $$y_1=x_1t^{-\frac{1}{b_1}}, \ y_2=x_2t^{-\frac{1}{b_2}},\ \dots,\ y_d=x_dt^{-\frac{1}{b_d}}$$
that deforms $B_{\bf a,b}^d(t)$ into $B_{\bf a,b}^d$. The Jacobian
determinant is
$J(y)=t^{{\frac{1}{b_1}}+{\frac{1}{b_2}}+...+{\frac{1}{b_d}}}$.
 By the change of variables formula, we obtain $${\rm vol}(B_{\bf a,b}^d(t))=\int_{B_{\bf a,b}^d(t)}1dx=\int_{B_{\bf a,b}^d}J(y)dy=
 t^{\frac{1}{b_1}+\frac{1}{b_2}+\cdots+\frac{1}{b_d}}\,{\rm vol}(B_{\bf a,b}^d).$$ The proof of Lemma 3.2 is finished.
\end{proof}

\begin{lem}Let $\ab=\{a_j\}_{j\geq1}$ and $\bb=\{b_j\}_{j\geq1}$ be two sequences of positive numbers, and let $p_d=\max\{1,2b_1,2b_2,\dots,2b_d\}$.
Then for any $\xb,\yb\in \Bbb R^d$, we have
\begin{equation}\label{3.3}\Big(\sum_{j=1}^d a_j\vert x_j+y_j\vert^{2b_j}\Big)^\frac{1}{p_d}\leq\Big(\sum_{j=1}^d a_j\vert x_j\vert^{2b_j}\Big)^\frac{1}{p_d}+
\Big(\sum_{j=1}^d a_j\vert
y_j\vert^{2b_j}\Big)^\frac{1}{p_d}.\end{equation}
\end{lem}

\begin{proof}It follows from \cite[Lemma 3.2]{CW} that for
$\xb=(x_1,\dots,x_d), \ \yb=(y_1,\dots,y_d),\ \xb,\yb\in \Bbb
R^d$, we have
$$
\Big(\sum_{j=1}^d|x_j+y_j|^{2b_j}\Big)^{1/p_d}\le
\Big(\sum_{j=1}^d|x_j|^{2b_j}\Big)^{1/p_d}+\Big(\sum_{j=1}^d|y_j|^{2b_j}\Big)^{1/p_d}.
$$
If we replace $x_j,\ y_j, \ j=1,2,\dots,d$ by $a_j^{\frac
1{2b_j}}x_j,\ a_j^{\frac 1{2b_j}}y_j,\  j=1,2,\dots,d$ in the
above inequality, then we get
$$
\Big(\sum_{j=1}^d|a_j^{\frac 1{2b_j}}x_j+a_j^{\frac
1{2b_j}}y_j|^{2b_j}\Big)^{1/p_d}\le \Big(\sum_{j=1}^d|a_j^{\frac
1{2b_j}}x_j|^{2b_j}\Big)^{1/p_d}+\Big(\sum_{j=1}^d|a_j^{\frac
1{2b_j}}y_j|^{2b_j}\Big)^{1/p_d},
$$which gives \eqref{3.3}.
Lemma 3.3 is proved.
\end{proof}

\

\noindent {\it Proof of Theorem 2.2.}

\vskip 1mm

It follows from \eqref{2.5} that
$$a_n(I_d:W_2^{\ab,\bb}([0,1]^d)\rightarrow
L_2([0,1]^d))=W_{\ab,\bb,d}^*(n),$$where
$\{W_{\ab,\bb,d}^*(l)\}_{l=1}^\infty$ is the non-increasing
rearrangement of$$\Big\{\Big(1+\sum_{j=1}^d a_j\vert
k_j\vert^{2b_j}\Big)^{-\frac{1}{2}}\Big\}_{\kb=(k_1,\dots,k_d)\in
\Bbb Z^d}.$$

For $m\in \Bbb N$, let $C(m,\ab,\bb,d)$ denote the cardinality of
the set $$\Big\{\kb \ :\ \sum_{j=1}^d a_j\vert k_j\vert^{2b_j}\leq
m^{p_d},\ \kb\in \Bbb Z^d\Big\}.$$ where $p_d=\max\{1,2b_1,
2b_2,\dots,2b_d\}$. It follows that  for $n>C(m,\ab,\bb,d)$,
$$
a_n(I_d:W_2^{\ab,\bb}([0,1]^d)\rightarrow L_2([0,1]^d))\leq
(1+m^{p_d})^{-\frac{1}{2}}.$$ and for $n\leq C(m,\ab,\bb,d)$,
$$
a_n(I_d:W_2^{\ab,\bb}([0,1]^d)\rightarrow L_2([0,1]^d))\geq
(1+m^{p_d})^{-\frac{1}{2}}.$$

 For any $m\in \Bbb N$, let $Q_k$ be a cube with center $\kb$, sides parallel to the axes and side-length $1$.
For$$\xb\in\bigcup\limits_{\substack{{\bf k}\in\Bbb
Z^d\\\sum_{j=1}^{d}a_j\vert k_j\vert^{2b_j}\leq m^{p_d}}}Q_{\bf
k},$$ there exists a $\kb\in\Bbb Z^d$ such that $\sum_{j=1}^d
a_j\vert k_j\vert^{2b_j}\leq m^{p_d}$ and $\xb\in Q_\kb$. It
follows from  the definition of $Q_\kb$ that $$\vert
x_j-k_j\vert\leq\frac{1}{2},\ \ j=1,2,\dots,d.$$ By \eqref{3.3} we
have
\begin{align*}\Big(\sum_{j=1}^d a_j\vert x_j\vert^{2b_j}\Big)^\frac{1}{p_d}&\leq\Big(\sum_{j=1}^d a_j\vert x_j-k_j\vert^{2b_j}\Big)^\frac{1}{p_d}+
\Big(\sum_{j=1}^d a_j\vert
k_j\vert^{2b_j}\Big)^\frac{1}{p_d}\\&\leq\Big(\sum_{j=1}^d
a_j2^{-2b_j}\Big)^\frac{1}{p_d}+m.
\end{align*}
It follows that
\begin{equation}\label{3.4}\bigcup\limits_{\substack{{\bf
k}\in\Bbb Z^d\\\sum_{j=1}^{d}a_j\vert k_j\vert^{2b_j}\leq
m^{p_d}}}Q_{\bf k}\ \subset\ B_{\ab,2\bb}^d
\Big(\Big(m+\big(\sum_{j=1}^d
a_j2^{-2b_j}\big)^\frac{1}{p_d}\Big)^{p_d}\Big).\end{equation}
Using the same technique, we get
\begin{equation}\label{3.5}B_{\ab,2\bb}^d\Big(\Big(m-\big(\sum_{j=1}^d a_j2^{-2b_j}\big)^\frac{1}{p_d}\Big)^{p_d}_+\Big)
\ \subset\ \bigcup\limits_{\substack{{\bf k}\in\Bbb
Z^d\\\sum_{j=1}^{d}a_j\vert k_j\vert^{2b_j}\leq m^{p_d}}}Q_{\bf k}
,\end{equation} where $a_+$ is equal to $a$ if $a>0$ and 0 if
$a<0$. We note that the volume of the set
$${\rm vol}\,\Bigg(\ \bigcup\limits_{\substack{{\bf k}\in\Bbb
Z^d\\\sum_{j=1}^{d}a_j\vert k_j\vert^{2b_j}\leq m^{p_d}}}Q_{\bf
k}\Bigg)=\#\Big\{\kb \ :\ \sum_{j=1}^d a_j\vert
k_j\vert^{2b_j}\leq m^{p_d},\ \kb\in \Bbb Z^d\Big\}$$ is just
$C(m,\ab,\bb,d)$, where $\# A$ denotes the number of elements in a
set $A$. We set
$$C_{\ab,\bb,d}=\Big(\sum_{j=1}^d a_j2^{-2b_j}\Big)^\frac{1}{p_d}.$$ By \eqref{3.5}, \eqref{3.4}, and \eqref{3.2} we get
\begin{equation}\label{3.6}(m-C_{\ab,\bb,d})_+^{\frac{p_d}{g_d(2\bb)}}{\rm vol}(B_{\ab,2\bb}^d)
\leq C(m,\ab,\bb,d)\leq(m+C_{\ab,2\bb,d})^{\frac{p_d}{g_d(2\bb)}},
\end{equation}where $g_d(2\bb)=2g_d(\bb)=\frac{2}{\frac{1}{b_1}+\frac{1}{b_2}+\cdots+\frac{1}{b_d}}$. We write
$$a_n(I_d)=a_n(I_d:W_2^{\ab,\bb}([0,1]^d)\rightarrow
L_2([0,1]^d)).$$ We also write
$$A(m,\ab,\bb,d)=(m+C_{\ab,\bb,d})^{\frac{p_d}{2g_d(\bb)}}\,{\rm
vol}(B_{\ab,2\bb}^d),$$ and
$$B(m,\ab,\bb,d)=(m-C_{\ab,\bb,d})_+^{\frac{p_d}{2g_d(\bb)}}\,{\rm
vol}(B_{\ab,2\bb}^d).$$

On the one hand, we suppose that $A(m,\ab,\bb,d)<n\leq
A(m+1,\ab,\bb,d)$. By \eqref{3.6} we get $$n>C(m,\ab,\bb,d),$$
which implies $$a_n(I_d)\leq(1+m^{p_d})^{-\frac{1}{2}}.$$ It
follows that
\begin{align*}n^{g_d(\bb)}a_n(I_d)&\leq\frac{(A(m+1,\ab,\bb,d))^{g_d(\bb)}}{(1+m^{p_d})^{\frac{1}{2}}}\\
&\leq\frac{(m+1+C_{\ab,\bb,d})^{\frac{p_d}2}\,({\rm vol}
(B_{\ab,2\bb}^d))^{g_d(\bb)}}{(1+m^{p_d})^{\frac{1}{2}}}
\end{align*}
Obviously, we have
$$\lim_{m\to\infty}\frac{(m+1+C_{\ab,\bb,d})^{\frac{p_d}{2}}}{(1+m^{p_d})^{\frac{1}{2}}}=1,$$
which implies that
\begin{equation}\label{3.7}\lim_{n\to\infty} n^{g_d(\bb)}a_n(I_d)\leq({\rm vol}
(B_{\ab,2\bb}^d))^{g_d(\bb)}.\end{equation}

On the other hand, we suppose that $B(m,\ab,\bb,d)\le n<
B(m+1,\ab,\bb,d)$ as $n$ is sufficiently large. By \eqref{3.6} we
get $$n< C(m+1,\ab,\bb,d),$$ which implies
$$a_n(I_d)\geq(1+(m+1)^{p_d})^{-\frac{1}{2}}.$$ It follows that
\begin{align*}n^{{g_d(\bb)}}a_n(I_d)&\geq\frac{(B(m,\ab,\bb,d))^{{g_d(\bb)}}}{(1+(m+1)^{p_d})^{\frac{1}{2}}}\\
&\geq\frac{(m-C_{\ab,\bb,d})^{\frac{p_d}{2}}}{(1+(m+1)^{p_d})^{\frac{1}{2}}}\,({\rm
vol}(B_{\ab,2\bb}^d))^{g_d(\bb)}.
\end{align*}
We have
$$\lim_{m\to\infty}\frac{(m-C_{\ab,\bb,d})^{\frac{p_d}{2}}}{(1+(m+1)^{p_d})^{\frac{1}{2}}}=1,$$
which implies that
\begin{equation}\label{3.8}\lim_{n\to\infty} n^{{g_d(\bb)}}a_n(I_d)\geq(vol(B_{\ab,2\bb}^d))^{{g_d(\bb)}}.\end{equation}
Combining \eqref{3.7} with \eqref{3.8},  we obtain \eqref{2.8}.

 Theorem 2.2 is proved. $\hfill\Box$

\section{Tractability results of weighted anisotropic embeddings}

In this section, we first give the proof of Theorem 2.7. Next
 we establish  the relationship between
tractability of  $I=\{I_d\}_{d\in\Bbb N}$ and EC-tractability of
${\rm APP}=\{{\rm APP}_d\}_{d\in\Bbb N}$ based on Theorem 2.7,
where $I_d$ and ${\rm APP}_d$ are given by \eqref{1.1} and
\eqref{1.2}. Finally we show Theorems 2.8 and 2.9.

\vskip 2mm

 \noindent{\it Proof of Theorem 2.7.}

\vskip 1mm

 For any $\varepsilon
\in(0,1), \ \ d\in \Bbb N,$ by \eqref{2.6-1}, \eqref{2.5}, and
\eqref{2.6} we have
\begin{align}n(\varepsilon,I_d)
&=\min\{n:a_{n+1}(I_d:W_2^{\ab,\bb}\rightarrow
L_2([0,1]^d))\leq\varepsilon\}\notag \\
&=\min\{n:W_{\ab,\bb,d}^*(n+1)\leq\varepsilon\},\label{4.1}\end{align}and
\begin{align}n(\varepsilon,{\rm APP}_d)&=\min\{n:a_{n+1}({\rm
APP}_d:H(K_{d,\ab,2\bb})\rightarrow
L_2([0,1]^d))\leq\varepsilon\}\notag\\ &=\min\{n :
\lambda_{d,n+1}\leq\varepsilon^2\},\label{4.2}\end{align} where
$\{W_{\ab,\bb,d}^*(l)\}_{l=1}^\infty$ is the non-increasing
rearrangement of$$\Big\{\Big(1+\sum_{j=1}^d a_j\vert
k_j\vert^{2b_j}\Big)^{-\frac{1}{2}}\Big\}_{\kb=(k_1,\dots,k_d)\in
\Bbb Z^d},$$ and $\{\lz_{d,k}\}_{k=1}^\infty$ is  the
non-increasing rearrangement of $$\{\omega_{\kb}\}_{\kb\in \Bbb
Z^d}=\Big\{ \oz^{\sum_{j=1}^d a_j\vert
k_j\vert^{2b_j}}\Big\}_{\kb=(k_1,\dots,k_d)\in \Bbb Z^d}$$with the
fixed $\oz\in (0,1)$.

 For any $\varepsilon_1\in(0,1)$, by  \eqref{4.2} we have
\begin{align*}n(\varepsilon_1,{\rm APP}_d)
&=\min\{n:\lambda_{d,n+1}\leq\varepsilon_1^2\}=\#\{j\in\Bbb N:\lambda_{d,j}>\varepsilon_1^2\}\\
&=\#\{\kb\in \Bbb Z^d:\oz^{\sum_{j=1}^d a_j\vert k_j\vert^{2b_j}}>\varepsilon_1^2\}\\
&=\#\{\kb\in\Bbb Z^d:\oz^{-\sum_{j=1}^da_j\vert k_j\vert^{2b_j}}<\varepsilon_1^{-2}\}\\
&=\#\Big\{\kb\in\Bbb Z^d:\sum_{j=1}^d a_j\vert
k_j\vert^{2b_j}<\frac{\ln \varepsilon_1^{-2}}{\ln \oz^{-1}}\Big\}.
\end{align*}
We set $$\varepsilon_2=\Big(\frac{\ln \varepsilon_1^{-2}}{\ln
\oz^{-1}}+1\Big)^{-\frac{1}{2}}.$$ Then $\vz_2\in (0,1)$ and
$$\varepsilon_2^{-2}-1= \frac{\ln \varepsilon_1^{-2}}{\ln
\oz^{-1}}.$$ By \eqref{4.1} we obtain
\begin{align*}n(\varepsilon_2,I_d)&=\min\{n:W_{\ab,\bb,d}^*(n+1)\leq\varepsilon_2\}
=\#\{l\in\Bbb N:W_{\ab,\bb,d}^*(l)>\varepsilon_2\}
\\ &=\#\Big\{\kb\in\Bbb Z^d: \Big(1+\sum_{j=1}^d a_j\vert
k_j\vert^{2b_j}\Big)^{-\frac{1}{2}}>\varepsilon_2\Big\}
\\&=\#\Big\{\kb\in\Bbb Z^d:\sum_{j=1}^d a_j\vert k_j\vert^{2b_j}<\varepsilon_2^{-2}-1\Big\}\\
&=\#\{\kb\in\Bbb Z^d:\sum_{j=1}^d a_j\vert
k_j\vert^{2b_j}<\frac{\ln \varepsilon_1^{-2}}{\ln \oz^{-1}}\Big\}.
\end{align*}It follows that
$$n(\varepsilon_1,{\rm APP}_d)=n(\varepsilon_2,I_d)=n\Big(\Big(\frac{\ln \varepsilon_1^{-2}}
{\ln \oz^{-1}}+1\Big)^{-\frac{1}{2}},I_d\Big),$$which gives
\eqref{2.11}. Using the same method we can prove \eqref{2.12}.

Theorem 2.7 is proved. $\hfill\Box$

\begin{thm} Consider the approximation problems $I=\{I_d\}$ and
${\rm APP}=\{{\rm APP}_d\}$ in the worst case setting, where $I_d$
and ${\rm APP}_d$ are given by \eqref{1.1} and \eqref{1.2}.  We
have

$(\romannumeral 1)$  for fixed  $s,t>0$,  $I$ is $(2s,t)$-WT iff
${\rm APP}$ is EC-$(s,t)$-WT;

$(\romannumeral 2)$    $I$ is UWT iff ${\rm APP}$ is EC-UWT;

$(\romannumeral 3)$ $I$ is QPT iff ${\rm APP}$  is EC-QPT;

$(\romannumeral 4)$ $I$ is PT iff ${\rm APP}$  is EC-PT;

$(\romannumeral 5)$ $I$ is SPT iff ${\rm APP}$  is EC-SPT.
\end{thm}
\begin{proof}$(\romannumeral 1)$ First we prove the sufficiency. Assume that $I$ is
$(2s,t)$-WT. For $\vz\in(0,1)$, we set
\begin{equation}\label{4.3}\varepsilon_1=\Big(\frac{\ln \varepsilon^{-2}}{\ln
\oz^{-1}}+1\Big)^{-\frac{1}{2}}\in(0,1).\end{equation}Then we have
$$d+\vz^{-1}\to\infty\ \ {\rm iff}\ \ d+\vz_1^{-1}\to\infty.$$

By the definition of $(s,t)$-WT we have
$$\lim_{d+\varepsilon_1^{-1}\to\infty}\frac{\ln n(\varepsilon_1 ,I_d)}{d^t+({\varepsilon_1^{-1}})^{2s}}=0.$$
It follows from \eqref{2.11} that
\begin{align*}\frac{\ln n(\varepsilon ,{\rm APP}_d)}{d^t+(1+ {\ln \varepsilon^{-1}})^s}
=\frac{\ln n(\vz_1, I_d)}{d^t+\vz_1^{-2s}}\ \frac {d^t+(\frac{\ln
\varepsilon^{-2}}{\ln \oz^{-1}}+1)^s}{d^t+(1+{\ln
\varepsilon^{-1}})^s}.
\end{align*}
Noting that
\begin{equation}\label{4.4}\vz_1^{-2}=\frac{\ln
\varepsilon^{-2}} {\ln \oz^{-1}}+1\le C_1 (1+\ln
\varepsilon^{-1}),\end{equation}where $C_1=\max\{1,\frac2{\ln
\oz^{-1}}\}$, we get
$$\frac{d^t+(\frac{\ln
\varepsilon^{-2}} {\ln \oz^{-1}}+1)^s}{d^t+(1+{\ln
\varepsilon^{-1}})^s}\le \frac{d^t+C_1^s(1+{\ln
\varepsilon^{-1}})^s}{d^t+(1+{\ln \varepsilon^{-1}})^s} \le
C_1^s.$$

It follows that
$$\lim_{d+\varepsilon^{-1}\to\infty}\frac{\ln n(\varepsilon ,
{\rm APP}_d)}{d^t+(1+{\ln \varepsilon^{-1}})^s}=0,$$ which yields
that ${\rm APP}$ is EC-$(s,t)$-WT.

Next we show the necessity. Suppose  that ${\rm APP}$ is
EC-$(s,t)$-WT. For $\vz\in(0,1)$, we set
\begin{equation}\label{4.5}\varepsilon_2=\oz^{\frac{\varepsilon^{-2}-1}{2}}\in(0,1).\end{equation}
Then we have  $$d+\vz^{-1}\to\infty \ \ {\rm iff}\ \
d+\vz_2^{-1}\to\infty.$$ By the definition of EC-$(s,t)$-WT we
have
 $$\lim_{d+\varepsilon_2^{-1}\to\infty}\frac{\ln n(\varepsilon_2,{\rm APP}_d)}{d^t+(1+{\ln \varepsilon_2^{-1}})^s}=0.$$
By \eqref{2.12}, we have
\begin{align*}\frac{\ln n(\varepsilon,I_d)}{d^t+({\varepsilon^{-1}})^{2s}}
=\frac{\ln n(\varepsilon_2,{\rm APP}_d)}{d^t+(1+{\ln
\varepsilon_2^{-1}})^s}\ \frac{d^t+
(1+\frac{\varepsilon^{-2}-1}{2}\ln\oz^{-1})^s}{d^t+({\varepsilon^{-1}})^{2s}}.
\end{align*}
Noting that\begin{equation}\label{4.6}
1+\ln\vz_2^{-1}=1+\frac{\varepsilon^{-2}-1}{2}\ln\oz^{-1}\le
1+\vz^{-2}\ln\oz^{-1}\le C_2\vz^{-2},\end{equation}where
$C_2=1+\ln\oz^{-1}$, we obtain
$$\frac{d^t+
(1+\frac{\varepsilon^{-2}-1}{2}\ln\oz^{-1})^s}{d^t+({\varepsilon^{-1}})^{2s}}\le
\frac{d^t+ C_2^s \,\vz^{-2s}}{d^t+\varepsilon^{-2s}}\le C_2^s.
$$ It follows that  $$\lim_{d+\varepsilon^{-1}\to\infty}\frac{\ln
n(\varepsilon ,I_d)} {d^t+({\varepsilon^{-1}})^{2s}}=0,$$ which
leads to that $I$ is $(2s,t)$-WT. Hence ({\it i}) is proved.
 \vskip 1mm
$(\romannumeral 2)$ We note that  $I$ is UWT iff  for all
$\alpha,\beta>0$, $I$ is $(\alpha,\beta)$-WT. By (i) we obtain
that $I$ is UWT iff
  for all
$\alpha,\beta>0$, ${\rm APP}$ is EC-$(\frac{\alpha}{2},\beta)$-WT.
This is equivalent to that ${\rm APP}$ is EC-UWT. Hence ({\it ii})
is proved.
 \vskip 1mm
$(\romannumeral 3)$ We first assume that $I$ is QPT. Then there
exist two constants $C,t>0$ such that for all $d\in \Bbb N, \ \va
\in(0,1)$,
\begin{equation*}
n(\va ,I_d)\leq C\exp[t(1+\ln\va ^{-1})(1+\ln d)].
\end{equation*}
By Theorem 2.7 and \eqref{4.4} we have
\begin{align*} n(\varepsilon,{\rm APP}_d)&=n\Big(\vz_1,I_d\Big)\leq C\exp[t(1+\ln\va_1 ^{-1})(1+\ln d)]
\\&\leq C \exp\Big\{t\Big[1+\frac12\ln(C_1(1+\ln\vz^{-1}))\Big]\big(1+\ln d\big)\Big\}
\\&\le C \exp\Big\{t_1\,\big[1+\ln(1+\ln
\varepsilon^{-1})\big]\big(1+\ln d\big)\Big\},
\end{align*}
where $\vz_1$ is given by \eqref{4.3}, $t_1=t(1+\frac{\ln C_1}2)$,
$C_1=\max\{1,\frac2{\ln \oz^{-1}}\}$. This means that ${\rm APP}$
is EC-QPT.

Next, we assume  that  ${\rm APP}$ is EC-QPT. We want to show that
$I$ is QPT. There exist two constants $C,t>0$ such that for all
$d\in \Bbb N, \ \va \in(0,1)$,
\begin{equation*}
n(\va ,{\rm APP}_d)\leq C\exp\{t[1+\ln(1+\ln \va ^{-1})](1+\ln d)\}.
\end{equation*}
According to Therorem 2.7 and \eqref{4.6}, we obtain
\begin{align*}n(\varepsilon,I_d)&=n(\vz_2,{\rm APP}_d)\le  C\exp\big\{t[1+\ln(1+\ln \va_2 ^{-1})](1+\ln d)\big\}\\&
\le  C\exp\big \{t[1+\ln(C_2\vz^{-2})](1+\ln d)\big\} \\&\le
C\exp\big[t_2(1+\ln\vz^{-1}))(1+\ln d)\big],
\end{align*}where $\vz_2$ is given by \eqref{4.5}, $t_2=t\max\{1+\ln C_2,2\}$, $C_2=1+\ln\oz^{-1}$.
This means that $I$ is QPT. Hence ({\it iii}) is proved.

 \vskip 1mm

$(\romannumeral 4)$ First we assume that $I$ is PT. There exist
non-negative numbers $C, p$ and $q$ such that for all $d\in \Bbb
N, \ \va \in(0,1)$,
\begin{equation*}
n(\va ,I_d)\leq Cd^q\va ^{-p}.
\end{equation*}
By Theorem 2.7, \eqref{4.3}, \eqref{4.4} we have
\begin{align*}n(\varepsilon,{\rm APP}_d)=n(\vz_1,I_d)
\leq Cd^q\vz_1^{-p}\le CC_1^{p/2}d^q(1+\ln\vz^{-1})^{p/2},
\end{align*}
which means  that {\rm APP} is EC-PT.

Next, suppose that  APP is EC-PT. There exist non-negative numbers
$C, p$ and $q$ such that for all $d\in \Bbb N, \ \va \in(0,1)$,
\begin{equation*}
n(\va ,{\rm APP}_d)\leq Cd^q(1+\ln \va ^{-1})^p.
\end{equation*}
By Theorem 2.7, \eqref{4.5}, \eqref{4.6} we have
\begin{align*}n(\varepsilon,I_d)=n(\vz_2,{\rm APP}_d)\le Cd^q(1+\ln \va_2 ^{-1})^p\leq C
C_2^pd^q\vz^{-2p},
\end{align*}
which implies that $I$ is PT. Hence ({\it iv}) is proved.
 \vskip 1mm

$(\romannumeral 5)$ The proof is the same as the one of
$(\romannumeral 4)$.

\vskip 1mm The proof of Theorem 4.1 is finished.
\end{proof}

\vskip 2mm

 \noindent{\it Proof of Theorem 2.8.}

\vskip 1mm

According to  \cite{IKPW} and \cite{Wangh}, we have  the following
EC-tractability  results  of
  ${\rm APP}$:

  \vskip 2mm

 $\bullet$  EC-SPT holds iff EC-PT holds iff
$$\sum_{j=1}^\infty b_j^{-1}<\infty  \qquad {\rm and}\qquad\underset{j\to\infty}{\underline{\lim}}\frac{\ln a_j}{
j}>0.$$

\vskip 1mm

 $\bullet$  EC-QPT holds iff $$\sup_{d\in \Bbb N}\frac{\sum_{j=1}^d b_j^{-1}}{1+\ln d}<\infty\qquad {\rm and}\qquad
 \underset{j\to\infty}{\underline{\lim}}\frac{(1+\ln j)\ln a_j}{
j}>0.$$

\vskip 1mm

 $\bullet$  EC-UWT holds iff $$\lim_{j\to\infty}\frac{\ln a_j}{\ln j}=\infty.$$

\vskip 1mm

$\bullet$ EC-$(s,t)$-WT with $\max(s,t)>1$ always holds.

\vskip 3mm

 $\bullet$ EC-WT  holds  iff  $$\lim\limits_{j\to \infty}a_j=\infty.$$

\vskip 1mm

 $\bullet$ EC-$(1,t)$-WT with $t<1$ holds iff
\begin{equation*}\lim\limits_{j\to\infty}\frac{ a_j}{\ln j}=\infty.\end{equation*}

\vskip 1mm

 $\bullet$  EC-$(s,t)$-WT with $s<1$ and $t\leq 1$  holds iff
\begin{equation*}\lim\limits_{j\to\infty}\frac{ a_j}{j^{(1-s)/s}}=\infty.\end{equation*}

\vskip 1mm

Hence, Theorem 2.8 follows from  Theorem 4.1 and the above results
immediately.

 $\hfill\Box$

\vskip 2mm

 \noindent{\it Proof of Theorem 2.9.}

\vskip 1mm

We note that if $\tilde \bb=\{\tilde b_j\}$, $\tilde
b_j=(2\pi)^{2b_j},\ j\in\Bbb N$, then
$$W_2^{\bb}([0,1]^d)=W_2^{\tilde\bb,\bb}([0,1]^d).$$It follows from
Theorem 2.8 that

\noindent(1) $\tilde I$ is  SPT  iff $\tilde I$ is PT  iff
$$\sum\limits_{j=1}^\infty b_j^{-1}<\infty\qquad {\rm and}\qquad\underset{j\to\infty}{\underline{\lim}}\frac{b_j}{
j}>0;$$

 \noindent(2)  $\tilde I$ is QPT  iff  $$\sup\limits_{d\in \Bbb
N}\frac{\sum_{j=1}^d b_j^{-1}}{1+\ln d}<\infty  \qquad {\rm
and}\qquad
 \underset{j\to\infty}{\underline{\lim}}\frac{(1+\ln j)\,b_j}{
j}>0;$$

 \noindent (3) ({\it iii}), ({\it iv}), ({\it v}), ({\it
vi}), ({\it vii}), ({\it viii}) hold.

\

Hence, in order to show  Theorem 2.9, it suffices to prove that

\vskip 2mm

(a) if $ \sum\limits_{j=1}^\infty b_j^{-1}<\infty$, then  we have
$\underset{j\to\infty}{\underline{\lim}}\frac{b_j}{ j}>0;$

\vskip 2mm

(b) if $\sup\limits_{d\in \Bbb N}\frac{\sum_{j=1}^d
b_j^{-1}}{1+\ln d}<\infty$, then
$\underset{j\to\infty}{\underline{\lim}}\frac{(1+\ln j)\,b_j}{
j}>0.$

\vskip 3mm

First  we prove (a). Assume that $B:= \sum\limits_{j=1}^\infty
b_j^{-1}<\infty$. We have
$$B\ge \sum\limits_{j=1}^k
b_j^{-1}\ge kb_k^{-1}.$$ It follows that
$$\frac{b_k}k\ge \frac1B,$$which yields that $$\underset{k\to\infty}{\underline{\lim}}\frac{b_k}{
k}\ge \frac1B>0.$$ Hence (a) holds.

Next we show (b).  Assume that $D:=\sup\limits_{d\in \Bbb
N}\frac{\sum_{j=1}^d b_j^{-1}}{1+\ln d}<\infty$. We have
$$D\ge \frac{\sum_{j=1}^k b_j^{-1}}{1+\ln k}\ge \frac{k b_k^{-1}}{1+\ln k}.
$$ It follows that
$$\frac{(1+\ln k)b_k}k\ge \frac1D,$$which yields that $$\underset{k\to\infty}{\underline{\lim}}\frac{(1+\ln k)\,b_k}{
k}\ge \frac1D>0.$$ Hence (b) holds.

This completes the proof of Theorem 2.9. $\hfill\Box$

\

\noindent{\bf Acknowledgment}  This work was supported by the
National Natural Science Foundation of China (Project no.
11671271) and
 the  Natural Science Foundation of Beijing Municipality (1172004).


\begin{thebibliography}{99}

\bi{CW} J. Chen, H. Wang, Preasymptotics and asymptotics of
approximation numbers of anisotropic Sobolev embeddings,  J.
Complexity, 39 (2017) 94-110.

\bi{CW1}J. Chen, H. Wang, Approximation numbers of  Sobolev and
Gevrey type embeddings on the sphere and on the ball --
Preasymptotics, asymptotics, and tractability,  J. Complexity 50
(2019) 1-24.

\bibitem{CKS} F. Cobos, T. K\"{u}hn, W. Sickel, Optimal approximation of multivariate periodic Sobolev functions in the sup-norm, J. Funct. Anal. 270 (11) (2016) 4196-4212.



\bibitem{DKPW}J. Dick, P. Kritzer, F. Pillichshammer, H. Wo\'zniakowski,
Approximation of analytic functions in Korobov spaces, J.
Complexity 30 (2014) 2-28.

\bibitem{DLPW} J. Dick, G. Larcher, F. Pillichshammer, H. Wo\'zniakowski,
 Exponential convergence and tractability of multivariate integration for Korobov spaces, Math. Comp. 80 (2011) 905-930.



\bibitem{IKPW}C. Irrgeher, P. Kritzer, F. Pillichshammer, H. Wo\'zniakowski,
 Tractability of multivariate approximation defined over Hilbert spaces with exponential weights, J. Approx. Theory 207 (2016) 301-338.

 \bibitem{KPW0} P. Kritzer, F. Pillichshammer, H. Wo\'zniakowski, Multivariate integration of infinitely many times differentiable functions
in weighted Korobov spaces, Math. Comp. 83 (2014) 1189-1206.

 \bibitem{KPW}P. Kritzer, F. Pillichshammer, H. Wo\'zniakowski, Tractability of multivariate analytic problems, in: P. Kritzer,
H. Niederreiter, F. Pillichshammer, A. Winterhof (Eds.), Uniform
Distribution and Quasi-Monte Carlo Methods. Discrepancy,
Integration and Applications, De Gruyter, Berlin, 2014, pp.
147-170.
\bibitem{K} T. K\"uhn, A lower estimate for entropy numbers, J. Approx. Theory, 110 (2001)
120-124.

\bibitem{KMU} T. K\"uhn, S. Mayer,  T. Ullrich, Counting via entropy: new preasymptotics
for the approximation numbers of Sobolev embeddings, SIAM J.
Numer. Anal. 54 (6) (2016) 3625-3647.
\bibitem{KSU1} T. K\"uhn, W. Sickel, T. Ullrich, Approximation numbers of Sobolev embeddings-Sharp constants and tractability, J. Complexity 30 (2014) 95-116.
\bibitem{KSU2} T. K\"uhn, W. Sickel, T. Ullrich,  Approximation of  mixed order Sobolev functions on the $d$-torus: asymptotics, preasymptotics,
and $d$-dependence,  Constr. Approx. 42(3) (2015) 353-398.
\bibitem{LX1} Y. Liu, G. Xu,  A note on tractability of multivariate analytic
problems, J. Comlexity, 34 (2016) 42-49.
\bibitem{NW1} E. Novak, H. Wo\'zniakowski, Tractablity  of Multivariate Problems, Volume I: Linear Information, EMS, Z\"urich, 2008.
\bibitem{NW2} E. Novak, H. Wo\'zniakowski, Tractablity  of Multivariate Problems, Volume II: Standard Information for Functionals, EMS, Z\"urich, 2010.
\bibitem{NW3} E. Novak, H. Wo\'zniakowski, Tractablity  of Multivariate Problems, Volume III: Standard Information for Operators, EMS, Z\"urich, 2012.
\bibitem{PP} A. Papageorgiou, I. Petras, A new criterion for tractability of multivariate problems, J. Complexity 30 (2014) 604-619.
 \bibitem{Pi1} A. Pietsch, Operator Ideals, North-Holland, Amsterdam, 1980.
 \bibitem{Pi2} A. Pietsch, Eigenvalues and s-Numbers, Cambridge University Press, Cambridge, 1987.
  \bi{P} A. Pinkus, n-Widths in Approximation Theory,  Ergeb.
Math. Grenzgeb., vol. 3.7, Springer, Berlin, 1985.
\bibitem{X2}G. Xu, Exponential convergence-tractability of general linear problems in the average case setting, J. Complexity 31 (2015)
617-636.
\bibitem{Wangh}H. Wang, A note about EC-$(s,t)$-weak tractability  of  multivariate  approximation with analytic Korobov
kernels, to appear in J. Complexity. Also see
http://arxiv.org/abs/1808.01470.
\bibitem{Wangx} X. Wang, Volumes of generalized unit balls, Math.Mag. 78 (5) (2005) 390-395.

\bibitem{WW} A. Werschulz, H. Wo\'zniakowski, Tractability of multivariate
approximation over weighted standard Sobolev spaces, J.
Complexity, in press (2018).



\end{thebibliography}
\end{document}